\documentclass[
11pt,  reqno]{amsart}
\usepackage{amsmath,amssymb,amscd,amsthm}

\usepackage{color}
\usepackage{hyperref}

\usepackage{cases}

\usepackage{bm}

\allowdisplaybreaks[3]

\newtheorem{theorem}{Theorem} [section]

\newtheorem{lemma}[theorem]{Lemma}
\newtheorem{proposition}[theorem]{Proposition}
\newtheorem{remark}[theorem]{Remark}
\newtheorem{definition}[theorem]{Definition}

\DeclareMathOperator*{\esssup}{ess \ sup}

\newcommand{\Z}{\mathbb{Z}}
\newcommand{\R}{\mathbb{R}}
\newcommand{\C}{\mathbb{C}}
\newcommand{\T}{\mathbb{T}}

\newcommand{\N}{\mathbb{N}}

\newcommand{\TT}{\mathcal{T}}

\newcommand{\Ul}{\mathcal U}
\newcommand{\Vl}{\mathcal V}
\renewcommand{\H}{\mathcal H}

\newcommand{\uu}{\mathfrak u}
\newcommand{\fv}{\mathfrak v}

\newcommand{\dl}{\delta}

\newcommand{\eps}{\varepsilon}

\newcommand{\G}{\Gamma}

\newcommand{\Ld}{\Lambda}
\newcommand{\s}{\sigma}

\newcommand{\ft}{\widehat}
\newcommand{\Ft}{{\mathcal{F}}}
\newcommand{\wt}{\widetilde}
\newcommand{\cj}{\overline}
\newcommand{\dx}{\partial_x}

\newcommand{\dt}{\partial_t}

\renewcommand{\l}{\ell}

\newcommand{\les}{\lesssim}

\newcommand{\jb}[1]
{\langle #1 \rangle}

\newcommand{\nz}{P_{\neq 0}}

\numberwithin{equation}{section}
\numberwithin{theorem}{section}

\let\Re=\undefined\DeclareMathOperator*{\Re}{Re}
\let\Im=\undefined\DeclareMathOperator*{\Im}{Im}

\title[Nowhere continuity of the flow map of dNLS system on $\T$]
{Nowhere continuity of the flow map of \\ an integrable derivative nonlinear Schr\"odinger system
on the torus}

\author[T.~Kondo]
{Toshiki Kondo}

\address{
Toshiki Kondo\\
Department of Mathematics\\\
Graduate School of Advanced Science and Engineering\\
Hiroshima University\\
1-3-1 Kagamiyama,
Higashi-Hiroshima, 739-8526
Japan}
\email{d263049@hiroshima-u.ac.jp}

\subjclass[2020]{35Q55}

\keywords{Schr\"odinger system;
integrable;
well-posedness;
ill-posedness;
non-existence;
}

\begin{document}

\begin{abstract}
We consider a derivative nonlinear Schr\"odinger system
called the Chen-Lee-Liu type system on the torus.
This system is known as a completely integrable system.
We prove the flow map fails to be continuous at every point 
in the Sobolev space $H^s(\T) \times H^s(\T)$.
Moreover, we establish an additional condition required for the flow map to be continuous.
For the discontinuity, we take a sequence converging to the initial data for which the corresponding solutions do not exist. 
\end{abstract}

\maketitle



\section{Introduction}

We consider the following Cauchy problem for the derivative nonlinear Schr\"odinger system on the torus: 
\begin{equation}
\left\{
\begin{aligned}
&i \dt u - \dx^2 u - i u v \dx u = 0, && (t,x) \in \R \times \T,
\\
&i \dt v + \dx^2 v - i u v \dx v = 0, && (t,x) \in \R \times \T,
\\
&(u, v)|_{t=0} = (u_0, v_0), && x \in \T,
\end{aligned}
\right.
\label{CLL}
\end{equation}
where $\T : = \R / 2 \pi \Z$, $u$ and $v$ are $\C$-valued unknown functions. 
The initial data $(u_0, v_0)$ is given in the Sobolev space
\[
\H^s(\T) := H^s(\T) \times H^s(\T)
\] 
equipped with the norm
\[
\| (f ,g) \|_{\H^s} := \|f\|_{H^s} + \|g\|_{H^s}. 
\]

Chen-Lee-Liu \cite{CLL79} developed a method for determining the integrability 
of a given nonlinear Hamiltonian system. 
By employing this framework, 
it follows that \eqref{CLL} is a completely integrable system. See also \cite{AKNS74}.

From the completely integrability, we obtain that \eqref{CLL} has infinite conserved quantities.
For example,
\begin{align}
M(u,v) &:= \int_\T u(t,x) v(t,x) dx,
\label{conser}
\\
E_1(u,v) &:= \int_\T  u(t,x) \dx v(t,x) dx,
\notag
\\
E_2(u,v) &:= \int_\T \Big( i \dx u (t,x) \dx v (t,x) - u(t,x)^2 v(t,x) \dx v(t,x) \Big) dx,
\notag
\end{align}
and so on.

By substituting $\cj u$ to $v$ in \eqref{CLL}, the following derivative nonlinear Schr\"odinger equation (DNLS) is obtained:
\begin{equation}
i \dt u + \dx^2 u - i |u|^2 \dx u = 0.
\label{dNLS}
\end{equation}
This is a fundamental model in ultrashort self-steepening optical pulses \cite{MMW07}. 
Let us review previous studies about \eqref{dNLS}.
In the real line setting,
\eqref{dNLS} is often studied together with the following derivative nonlinear Schr\"odinger equation:
\begin{equation}
i \dt u + \dx^2 u - i \dx (|u|^2 u) = 0,
\label{dNLS2}
\end{equation}
since \eqref{dNLS} is rewritten to \eqref{dNLS2} through the gauge transformation 
\cite{Ku84, Ha93}.
Since there are many studies on \eqref{dNLS} and \eqref{dNLS2}, we refer here only to some of them,
see \cite{TsFu80, TsFu81, Ha93, Ta99, GuWu17, JLPS20, BaPe22, KNV23, HKNV26}.

Even in the periodic setting, 
\eqref{dNLS} and \eqref{dNLS2} have also been extensively studied.
See \cite{TsFu80, TsFu81, He06, GrHe08, AmSi15}.
Moreover, there has also been some studies on the global well-posedness for \eqref{dNLS2}.
For instance, see  \cite{Mo17, KNV23, BaPe24}.
In contrast to the case of the real line setting, 
applying a gauge transformation does not make \eqref{dNLS2} exactly coincide with \eqref{dNLS}.
See Remark \ref{rem:gauge}.

Recently, nonlocal models associated with integrable systems have attracted much attention, 
and nonlocal versions of the modified Korteweg-de Vries equation \cite{AbMu17}, 
the nonlinear Schr\"odinger equation \cite{AbMu13}, and 
the derivative nonlinear Schr\"odinger equations \cite{AbMu17} have been studied.
The nonlocal equation associated with \eqref{dNLS},
\begin{equation}
i \dt u + \dx^2 u - i u u^\ast \dx u = 0,
\label{NdNLS}
\end{equation} 
was derived by Shi-Shen-Zhao \cite{SSZ19}, where $u^\ast(t,x) = \cj {u (t,-x)}$. 
This derivation is carried out by applying the reduction $v = u^\ast$ to the equation \eqref{CLL}.
Moreover, using the same method, 
they listed some nonlocal models based on \eqref{CLL}.

As a well-posedness study of \eqref{NdNLS},
Chen-Lu-Wang \cite{CLW23} proved the global well-posedness for \eqref{NdNLS} 
in very rough function spaces that contain $H^s(\R)$ for every $s \in \R$ and whose Fourier transforms are supported in a half-space.
Moreover, Nakanishi-Wang \cite{NaWa25} showed the global well-posedness of general evolution equations including \eqref{NdNLS} for distributions on the Fourier half space
for both the real line setting and the periodic setting.

In this paper, 
we prove that the flow map of \eqref{CLL} is discontinuous at every point in $\H^s(\T)$.
Moreover, 
we show an additional condition required for the flow map of \eqref{CLL} to be continuous. 
Before stating the main result, we define some notation.
\begin{definition}
\label{def:sol}
Let $s>\frac 32$, $T > 0$, and $\phi \in H^s(\T)$.
We say that $(u,v)$ is a solution to \eqref{CLL} in $\H^s(\T)$ on $[0,T]$
if $(u,v)$ satisfies the following two conditions:

\begin{enumerate}
\item
$(u,v) \in C ([0,T]; \H^s(\T))$;

\item
For any $\chi \in C_c^\infty([0,T) \times \T)$,%
\footnote{Here, $C_c^\infty([0,T) \times \T)$ denotes the space of smooth functions with compact support in $[0,T) \times \T$.
}
we have
\begin{align*}
&- i \int_0^T \int_\T u (t,x) \dt \chi(t,x) dx dt
- i \int_\T u_0 (x) \chi (0,x) dx
\\
&\quad
- \int_0^T \int_\T u(t,x) \dx^2 \chi(t,x) dx dt
\\
&=
i \int_0^T \int_\T u(t,x) v(t,x) \dx u(t,x) \chi (t,x) dx dt
\end{align*}
and
\begin{align*}
\\
&- i \int_0^T \int_\T v (t,x) \dt \chi(t,x) dx dt
- i \int_\T v_0 (x) \chi (0,x) dx
\\
&\quad
+ \int_0^T \int_\T v(t,x) \dx^2 \chi(t,x) dx dt
\\
&=
i \int_0^T \int_\T u(t,x) v(t,x) \dx v(t,x) \chi (t,x) dx dt.
\end{align*}
\end{enumerate}
A solution on $[-T,0]$ is defined in a similar manner.
\end{definition}
The second condition (ii) in Definition \ref{def:sol} means that 
$u$ satisfies \eqref{CLL} in the sense of distribution.

The following theorem is the main result in the present paper.

\begin{theorem}
\label{thm:IP}
Let $s > \frac 32$. 
The flow map of \eqref{CLL} is discontinuous at any point $(u_0, v_0) \in \H^s(\T)$.
Specifically,
for any $(u_0, v_0) \in \H^s(\T)$, 
there exists a sequence $\{ (u_{0,n}, v_{0,n}) \}_{n \in \N}$ converging to $(u_0, v_0)$
in $\H^s(\T)$ such that no corresponding solutions $(u_n, v_n)$ in $\H^s(\T)$ 
exist on either $[0,T]$ or $[-T,0]$.
\end{theorem}

The completely integrable systems like \eqref{CLL} have useful properties in general,
for example, the existence of infinitely many conserved quantities. 
Nevertheless, 
Theorem \ref{thm:IP} implies that, 
within the framework of Sobolev spaces, 
\eqref{CLL} is not well-posed.

On the other hand, 
\eqref{dNLS} is well-posed on $H^s(\T)$. 
Therefore, 
the Cauchy problem \eqref{CLL} may be well-posed 
when additional conditions are imposed on the space $\H^s(\T)$.
This is actually true, which is the next main result.

\begin{theorem}
\label{thm:WP}
Let $s > \frac 32$.
The Cauchy problem \eqref{CLL} is well-posed in 
\begin{equation}
\H_0^s (\T) := \{ (f,g) \in \H^s(\T) \mid \Im M(f,g) = 0 \},
\label{subset} 
\end{equation}
where $M$ is defined in \eqref{conser}.
More precisely, the followings hold:
\begin{enumerate}
\item For any $(u_0, v_0) \in \H_0^s(\T)$, 
there exists $T>0$ and a unique solution $(u,v) \in C([-T,T]; \H_0^s(\T))$ to \eqref{CLL}.
\item If $\{ (u_{0,n}, v_{0,n}) \}_{n \in \N} \subset \H_0^s(\T)$ converges to
$(u_0, v_0)$ in $\H_0^s(\T)$,
then we the corresponding solution $(u_n, v_n)$ 
with the initial data $(u_{0,n}, v_{0,n})$ converges to $(u,v)$ in $C([-T,T]; \H_0^s(\T))$. 
\end{enumerate}
\end{theorem}

Here, we give some examples of subsets of $\H_0^s(\T)$:
\begin{enumerate}
\item 
We set
\[
\H_{\text{conj}}^s(\T) := \{(f,g) \in \H^s(\T) \mid g = \cj f\}.
\]
Then, for any $(u,v) \in \H_{\text{conj}}^s(\T)$, we have
\[
\Im M(f,g) = \Im \int_\T |f(x)|^2 dx = 0.
\]
Hence, $\H_{\text{conj}}^s(\T) \subset \H_0^s(\T)$. 
Especially, DNLS \eqref{dNLS} is well-posed in $H^s(\T)$ for $s > \frac 32$. 
\item
Define
\[
\begin{aligned}
H_+^s(\T) &:= \{ f \in H^s(\T) \mid \ft f (n) = 0 \quad \text {for $n \le 0$} \},
\\
\H^s_+(\T) &:= H_+^s(\T) \times H_+^s(\T).
\end{aligned}
\]
Then, we obtain
\[
\begin{aligned}
\Im M(f,g) 
&= \Im \int_\T \Big (\sum_{n \in \N} \ft f(n) e^{inx} \Big) 
\Big( \sum_{m \in \N} \ft g(m) e^{imx} \Big) dx
\\
&=
\Im \int_\T \sum_{k=1}^\infty \Big( \sum_{j=1}^k \ft f(j) \ft g(k-j) e^{ikx} \Big) dx = 0. 
\end{aligned}
\]
Therefore, $\H^s_+(\T) \subset \H_0^s(\T)$.
In particular, $(f(\cdot), \cj {f(-\cdot)} ) \in \H^s_+(\T)$ provided with $f \in H_+^s(\T)$.
Hence, the nonlocal DNLS \eqref{NdNLS} is well-posed in $H_+^s(\T)$ for $s > \frac 32$.
\end{enumerate}

\begin{remark}
\rm
The regularity condition in Theorems \ref{thm:IP} and \ref{thm:WP} is not optimal.
For example, see \cite{He06}. 
However, 
since our main goal is to study the continuity of the flow map of \eqref{CLL} in $\H^s(\T)$,
we do not pursue the optimal regularity in this paper.
\end{remark}

To identify the most problematic nonlinear term in \eqref{CLL}, 
we consider the system of $\uu := \dx u$, $\fv := \dx v$. 
The equation for $(\uu, \fv)$ is given by
\[
\begin{aligned}
i \dt \uu - \dx^2 \uu &= i u v \dx \uu + i \uu v \uu + i u \fv \uu,
\\
i \dt \fv + \dx^2 \fv &= i u v \dx \fv + i \uu v \fv + i u \fv^2.  
\end{aligned}
\]  
Since the second and third terms on the right-hand side do not contain $\dx \uu$ and $\dx \fv$, 
they are harmless. 
On the other hand, 
the first term cannot be handled directly because a derivative loss exists.
Thus, we apply a gauge transformation to cancel out this term.
For the gauge transformation, see \cite{HaOz92, He06}.
However, since the gauge transformation cannot remove the constant term, 
it is necessary to impose a new condition that this term vanishes.
Specifically, by setting 
\[
U:= \exp \Big( - \frac i2 \dx^{-1} (uv) \Big) \uu, 
\quad 
V:= \exp \Big( \frac i2 \dx^{-1} (uv) \Big) \fv, 
\]
$(U, V)$ satisfies 
\begin{equation}
\begin{aligned}
i \dt U - \dx^2 U &= i \int_\T u v\; dx \cdot \dx U + \text{(other terms)},
\\
i \dt V + \dx^2 V &= i \int_\T u v \; dx \cdot \dx V + \text{(other terms)}.
\end{aligned}
\label{eqU}
\end{equation}
From the conservation of $\int_\T u v dx = M(u,v)$,
we have $\int_\T u v dx = \int_\T u_0 v_0 dx$.
Then,
if $\Im \int_\T u_0 v_0 \; dx \neq 0$,
the term $\Im \int_\T u_0 v_0 \; dx \cdot \dx U$ does not be canceled out from \eqref{eqU}.
Thus, 
the Cauchy–Riemann-type elliptic operator
$i \dt + a \dx$ ($a \in \R$) formally appears in \eqref{eqU}.
This suggests that \eqref{CLL} is ill-posed in $\H^s(\T)$.

For the proof of Theorem \ref{thm:IP},
we derive a sufficient condition on the initial data for the equation to have no solution 
by exploiting the smoothing effect of the Cauchy-Riemann-type elliptic operator.
See Theorem \ref{thm:NE} below.
Then, by constructing the initial data which the corresponding solution does not exist,
we get the non-existence result for \eqref{CLL}.
This approach is used in \cite{KiTs18, KoOk25b}.
Moreover, in this paper, we prove that for any $(u_0, v_0) \in \H^s(\T)$,   
there exists the sequence $\{ (u_{0,n}, v_{0,n}) \}_{n \in \N} \subset \H^s(\T)$
which converges to $(u_0, v_0)$ in $\H^s(\T)$ such that 
no corresponding solution to \eqref{CLL} exists.
See Proposition \ref{prop:app}.
As a result, we prove nowhere continuity of the flow map of \eqref{CLL}.

If we consider the equation \eqref{eqU} in $\H_0^s(\T)$ defined in \eqref{subset},
the first term on the right-hand side of \eqref{eqU} vanishes.
Moreover, one can see that ``(other terms)'' in \eqref{eqU} contains neither
$\dx \uu$ nor $\dx \fv$. See Lemma \ref{lem:funda} below. 
Therefore, 
if the regularity is sufficiently large,
we can apply the contraction mapping argument to \eqref{eqU}
and get a solution $(U, V) \in \H^{s-1}(\T)$ to \eqref{eqU} with the initial data 
$(U(0), V(0)) \in \H^{s-1}(\T)$.   
Then, we get a solution $(u,v) \in \H_0^s(\T)$ to \eqref{CLL} with
the initial data $(u_0, v_0) \in \H_0^s(\T)$.
Thus, we get Theorem \ref{thm:WP}.

%
\begin{remark}
\label{rem:gauge}
\rm
The equation \eqref{dNLS2} was obtained as the reduction $r = \cj q$ 
of the following completely integrable system:
\begin{equation}
\left\{
\begin{aligned}
&i \dt q - \dx^2 q - i \dx (q^2 r) = 0,
\\
&i \dt r + \dx^2 r - i \dx (q r^2) = 0.
\end{aligned}
\right.
\label{KN}
\end{equation}
Note that $M(q,r)$ defined in \eqref{conser} is a conserved quantity for \eqref{KN}.
The equations \eqref{CLL} and \eqref{KN} are known to be gauge equivalent.
More precisely, in the real line setting, $u$ and $v$ defined as
\begin{equation}
\begin{aligned}
u(t,x) &:= \exp \Big( \frac i2 \int_{-\infty}^x q(t,x) r(t,x) dx \Big) q(t,x),
\\
v(t,x) &:= \exp \Big(- \frac i2 \int_{-\infty}^x q(t,x) r(t,x) dx \Big) r(t,x),
\end{aligned}
\label{gaugeR}
\end{equation}
satisfy \eqref{CLL}. 
On the other hand, in the periodic setting, 
the function $\wt u$ and $\wt v$ obtained by replacing $\int_{-\infty}^\cdot$ in \eqref{gaugeR}
with $\dx^{-1}$ satisfies the corrected Chen-Lee-Liu type system: 
\begin{equation*}
\left\{
\begin{aligned}
&i \dt \wt u - \dx^2 \wt u = 2 i \Ft [\wt u \wt v](0) \dx \wt u 
+ \nz (\wt u \wt v) \dx \wt u + Q_1(\wt u , \wt v), 
\\
&i \dt \wt v + \dx^2 \wt v = 2 i \Ft [\wt u \wt v](0) \dx \wt v 
+ \nz (\wt u \wt v) \dx \wt v + Q_2(\wt u , \wt v),
\end{aligned}
\right.
\end{equation*}
where $Q_1$, $Q_2$ are the nonlinearities which do not contain $\dx \wt u$ and $\dx \wt v$.
Moreover, $M(\wt u, \wt v)$ is conserved since $q r = \wt u \wt v$.
Therefore, the method employed in the present paper also yields the same result for \eqref{KN}.
\end{remark}

\begin{remark}
\rm
In some cases, 
the existence of solutions can be proved even in function spaces 
that are not subsets of $\H_0^s(\T)$.
For example, 
the system \eqref{CLL} has an analytic solution with the analytic initial data $(u_0, v_0)$ 
even if $\Im M (u_0, v_0) \neq 0$. 
See Section 4 in \cite{KiTs18}. 
\end{remark}

This paper is organized as follows.
In Section \ref{SEC:IP}, 
we prove Theorem \ref{thm:IP} by using the gauge transformation and 
approximation argument. 
In Section \ref{SEC:WP}, we prove Theorem \ref{thm:WP} by the contraction mapping argument.

We conclude this section with notation.
We write the set of integers as $\Z$. 
Moreover, we denote the set of positive integers by $\N$.
We use $A \les B$ to denote $A \le C B$ with a constant $C>0$.

When $f \in L^1(\T)$, we set
\begin{align*}
\int_\T f(x) dx
:=
\frac 1{2\pi} \int_{-\pi}^\pi f(x) dx,
\quad
\ft f(k)
:= \int_\T f(x) e^{-ikx} dx
\end{align*}
for $k \in \Z$.
We also denote by $\ft f(k)$ or $\Ft [f](k)$ the Fourier coefficient of a distribution $f$ on $\T$.

We denote $H^s(\T)$ by the $L^2$-based Sobolev space on $\T$
equipped with the norm
\[
\| f \|_{H^s} :=
\Big( \sum_{k \in \Z} \jb{k}^{2s} |\ft f(k)|^2 \Big)^{\frac 12},
\]
where $\jb {k} = \sqrt{1 + |k|^2}$.
We also denote $\l^2$ by the space of square-summable sequences
equipped with the norm
\[
\| a \|_{\l^2} :=
\Big( \sum_{k \in \Z}|a(k)|^2 \Big)^{\frac 12}.
\]

For a distribution $f$ on $\T$,
we define
\begin{align*}
P_0 f
&:=
\ft f(0),
\qquad
P_{\neq 0} f
:=
f - P_0 f,
\\
\dx^{-1} f (x)
&:=
\sum_{k \in \Z \setminus \{0\}} \frac1{ik} \ft f(k) 
e^{ikx}.
\end{align*}
Note that $\dx (\dx^{-1} f) = \dx^{-1} (\dx f) = P_{\neq 0} f$.
Moreover, we also define the operators
\begin{align*}
P_+ f (x) := \sum_{k>0} \ft f(k) e^{ikx},
\quad
P_- f (x) := \sum_{k<0} \ft f(k) e^{ikx}.
\end{align*}

For $T>0$ and a Banach space $X$,
we denote the space of essentially bounded $X$-valued functions by $L^\infty([0,T]; X)$ equipped with the norm
\[
\| u \|_{L_T^\infty X}
:=
\esssup_{t \in [0,T]} 
\| u (t) \|_{X}.
\]


\section{The proof of theorem \ref{thm:IP}}
\label{SEC:IP}
In this section, we prove Theorem \ref{thm:IP}.
The following theorem is the main result in this section.

\begin{theorem}
\label{thm:NE}
Let $s> \frac 32$ and $T>0$.
Assume that
$u \in C([0,T]; H^s(\T))$ is a solution to \eqref{CLL}.
If the initial data $(u_0, v_0) \in \H^s(\T)$ satisfies 
$M(u_0, v_0) > 0$ (resp.\:$< 0$),
we have
\[
(P_- u_0, P_- v_0) \in \H^{s+\dl}(\T), \quad (resp.\: (P_+ u_0, P_+ v_0) \in \H^{s+\dl}(\T))
\]
for any  $\dl \in (0, 1)$.
The same conclusion holds for 
a solution 
\[
u \in C([-T,0]; H^s(\T))
\] by interchanging $P_-$ and $P_+$.
\end{theorem}
In the rest of this section, 
we only consider a solution forward in time.
Let $(u, v) \in C([0,T]; \H^s(\T))$ be the solution to \eqref{CLL}.
For $s > \frac 32$, $H^{s-1}(\T)$ is Banach algebra.
Hence, we have
\[
\begin{aligned}
&\| u v \dx u \|_{L_T^\infty H^{s-1}} + \| u v \dx v \|_{L_T^\infty H^{s-1}} 
\\
&\les \|u\|_{L_T^\infty H^{s-1}} \|v\|_{L_T^\infty H^{s-1}}
(\|u\|_{L_T^\infty H^s} + \|v\|_{L_T^\infty H^s} ). 
\end{aligned}
\]
Especially, $u v \dx u$ and $u v \dx v$ belong to $C([0,T];H^{s-1}(\T))$ for $s > \frac 32$.
This implies that $(u, v) \in C^1([0,T]; \H^{s-2}(\T))$ 
and
\begin{equation*}
\left\{
\begin{aligned}
&i \dt u - \dx^2 u - i u v \dx u = 0,
\\
&i \dt v + \dx^2 v - i u v \dx v = 0,
\end{aligned}
\right.
\end{equation*}
holds in $H^{s-2}(\T)$ for $t \in [0, T]$. 
Therefore, $M(u,v)$ is conserved.
Throughout this section,
we write $M_0 = M(u_0, v_0)$ for short.
 
Set
\begin{equation*}
\begin{aligned}
\uu := \dx u, \quad \fv := \dx v.  
\end{aligned}
\end{equation*}
Then, $(\uu, \fv)$ satisfies
\begin{equation}
\left\{
\begin{aligned}
&i \dt \uu - \dx^2 \uu = i u v \dx \uu + F_1,
\\
&i \dt \fv + \dx^2 \fv = i u v \dx \fv + F_2,
\end{aligned}
\right.
\label{CLLdx}
\end{equation}
where
\begin{equation}
\begin{aligned}
F_1 &= F_1 (u, v, \uu, \fv)
=
i \uu^2 v + i u \fv \uu,
\\
F_2 &= F_2 (u, v, \uu, \fv)
=
i \uu v \fv + i u \fv^2. 
\end{aligned}
\label{d:F} 
\end{equation}
To eliminate the problematic part in the first term on the right-hand side of \eqref{CLLdx},
we apply a gauge transformation.
Set
\begin{equation}
\Ld := - \frac i2  \dx^{-1} (u v),
\quad U := e^{-\Ld} \uu, \quad V:= e^\Ld \fv.  
\label{gauge}
\end{equation}
A direct calculation yields that
\[
\begin{aligned}
\dt U &= e^{-\Ld} ( - (\dt \Ld) \uu + \dt \uu ),
\\
\dx U &= e^{-\Ld} ( - (\dx \Ld) \uu + \dx \uu ),
\\
\dx^2 U &= e^{-\Ld} ( (\dx \Ld)^2 \uu - (\dx^2 \Ld) \uu - 2 (\dx \Ld) \dx \uu + \dx^2 \uu).
\end{aligned}
\]
Moreover, we have
\[
\begin{aligned}
\dt V &= e^\Ld ( (\dt \Ld) \fv + \dt \fv ),
\\
\dx V &= e^\Ld ( (\dx \Ld) \fv + \dx \fv ),
\\
\dx^2 V &= e^\Ld ((\dx \Ld)^2 \fv + (\dx^2 \Ld) \fv + 2 (\dx \Ld) \dx \fv + \dx^2 \fv).
\end{aligned}
\]
From \eqref{CLLdx}, \eqref{gauge}, and the conservation law, we obtain
\begin{equation*}
\left\{
\begin{aligned}
&i \dt U - \dx^2 U  
= i M_0 \dx U 
+ G_1,
\\
&i \dt V + \dx^2 V 
= i M_0 \dx V 
+ G_2,
\end{aligned}
\right.
\end{equation*}
where 
\begin{equation}
\begin{aligned}
G_1 &:= e^{-\Ld} F_1 - (i (\dt \Ld) + \Ft[u v](0) (\dx \Ld) - (\dx^2 \Ld) + (\dx \Ld)^2) U,
\\
G_2 &:= e^\Ld F_2 + (i (\dt \Ld) - \Ft[u v](0) (\dx \Ld) + (\dx^2 \Ld) + (\dx \Ld)^2) V.
\end{aligned}
\label{d:G}
\end{equation}
Moreover, we define the operator $\TT: C([0,T]; L^2(\T)) \to C([0,T]; L^2(\T))$ as
\[
\begin{aligned}
\TT f = f(t, x- (\Re M_0) t).
\end{aligned}
\]
Then, $(\TT U, \TT V)$ satisfy
\begin{equation*}
\left\{
\begin{aligned}
i \dt \TT U - \dx^2 \TT U 
&= 
- (\Im M_0) \dx \TT U + \TT G_1,
\\
i \dt \TT V + \dx^2 \TT V 
&= 
- (\Im M_0) \dx \TT V + \TT G_2,
\end{aligned}
\right.
\end{equation*}
where $G_1$, $G_2$ are defined in \eqref{d:G}.

In the following, let $\TT U$ and $\TT V$ be simply denoted as $U$ and $V$. 
Namely, we assume $(U, V)$ satisfy
\begin{equation}
\left\{
\begin{aligned}
i \dt U - \dx^2 U 
&= 
- (\Im M_0) \dx U + G_1,
\\
i \dt V + \dx^2  V 
&= 
- (\Im M_0) \dx V + G_2.
\end{aligned}
\right.
\label{CLLg4}
\end{equation}

We mention the fundamental estimate.
\begin{lemma}
\label{lem:funda}
Let $T > 0$, $s > \frac 32$. 
For $(u, v) \in C([0,T]; \H^s(\T))$,
there exists $C_s:= C(s, \|u\|_{L_T^\infty H^{s-1}}, \|v\|_{L_T^\infty H^{s-1}}, 
\|U\|_{L_T^\infty H^{s-1}}, \|V\|_{L_T^\infty H^{s-1}}) > 0$ such that
\begin{align*}
\|G_j(t)\|_{H^{s-1}} \le C_s
\end{align*}
for $j = 1,2$ and $t \in (0, T)$.
\end{lemma}

\begin{proof} 
Since the proofs in the case of $j=1$ and $j=2$ are the same, we only consider $j=1$.
For simplicity, we suppress the time dependence in this proof.

Fix $s > \frac 32$. 
Note that $H^{s-1} (\T)$ is Banach algebra.
From \eqref{gauge}, we obtain 
\begin{equation}
\begin{aligned}
\| e^{\pm \Ld} \|_{H^{s-1}} 
\les \sum_{\l = 0}^\infty \frac 1{\l !} \| \Ld \|_{H^{s-1}}^\l
\le
C_s. 
\label{p:eLd}
\end{aligned}
\end{equation}
Hence, we have
\begin{equation}
\begin{aligned}
\| \dx u \|_{H^{s-1}}
=
\|e^{\Ld} e^{-\Ld} \dx u\|_{H^{s-1}} 
\les
\|e^{\Ld}\|_{H^{s-1}} \|U\|_{H^{s-1}}
\le
C_s.
\end{aligned}
\label{es:dxu}
\end{equation}
Similarly, it holds that
\begin{equation}
\| \dx v \|_{H^{s-1}} \le C_s.
\label{es:dxv}
\end{equation}

A direct calculation with \eqref{gauge} yields that
\begin{equation}
\begin{aligned}
\dx \Ld = - \frac i2 \nz (uv), \quad
\dx^2 \Ld = - \frac i2 \dx (uv).
\end{aligned}
\label{dxLd}
\end{equation}
Moreover, we have
\begin{equation}
\begin{aligned}
\dt \Ld 
&= 
- \frac i2 \dx^{-1} \dt (u v)
\\
&=
- \frac i2 \dx^{-1} \big( \dx^2 u + i u v \dx u - \dx^2 v + i u v \dx v \big)
\\
&= - \frac i2 \big( \dx u - \dx v + i \dx^{-1} (u v (\dx u + \dx v) ) \big).
\end{aligned}
\label{dtLd}
\end{equation}
Since
\begin{equation*}
\begin{aligned}
\| u v \|_{H^{s-1}} 
\les \| u \|_{H^{s-1}} \|v \|_{H^{s-1}} 
\le C_s,
\end{aligned}
\end{equation*}
with \eqref{es:dxu}--\eqref{dtLd}, 
we have
\begin{equation}
\begin{aligned}
\|\dx \Ld\|_{H^{s-1}} 
&\le \|u v\|_{H^{s-1}} 
\le C_s,
\\
\| \dx^2 \Ld \|_{H^{s-1}}
&\le
\| (\dx u) v \|_{H^{s-1}} + \| u \dx v \|_{H^{s-1}} 
\le C_s,
\\
\| \dt \Ld \|_{H^{s-1}}
&\le
\| \dx u \|_{H^{s-1}} + \| \dx v \|_{H^{s-1}} 
\\
&\qquad+  \| \dx^{-1} (u v (\dx u + \dx v)) \|_{H^{s-1}}
\le
C_s.
\end{aligned}
\label{es:Ld}
\end{equation}
It follows from \eqref{p:eLd} and \eqref{d:F} that
\begin{equation}
\begin{aligned}
\|e^{-\Ld} F_1\|_{H^{s-1}}
&\les
\| e^{-\Ld} \|_{H^{s-1}} \|F_1\|_{H^{s-1}} 
\le
C_s. 
\end{aligned}
\label{es:F}
\end{equation}
Combining \eqref{d:G}, \eqref{es:Ld}, and \eqref{es:F}, we get the desired bound.
\end{proof} 

We recall the following estimate.
For the proof, see Lemma 2.3 in \cite{KoOk25b}.
\begin{lemma}
\label{lem:bili3}
Let $s \ge 0$ and $r>\frac 12$.
Then, we have
\[
\| P_+(f P_- g) \|_{H^s} + \| P_-(f P_+ g) \|_{H^s}
\les \| f \|_{H^s} \| g \|_{H^r}
\]
for any $f \in H^s(\T)$ and $g \in H^r(\T)$.
\end{lemma}

Using Lemmas \ref{lem:funda} and \ref{lem:bili3}, we prove Theorem \ref{thm:NE}.
\begin{proof}[Proof of Theorem \ref{thm:NE}]
We only consider the case $\Im M_0 > 0$, 
the case $\Im M_0 < 0$ follows from a slight modification.

Set
\[
\Ul := e^{it\dx^2} U, \quad \Vl:= e^{-it\dx^2} V.
\]
Note that
\[
\| \Ul(t) \|_{H^s} = \|U(t)\|_{H^s}, \quad \| \Vl(t) \|_{H^s} = \|V(t)\|_{H^s}
\]
for $t \in [0,T]$.
From \eqref{CLLg4}, we have
\begin{equation}
i \dt \ft \Ul(t,k)
=
- i (\Im M_0) k \ft \Ul(t,k) + \ft {G_1}(t,k).
\label{eqUlk}
\end{equation}
Solving \eqref{eqUlk}, we obtain
\begin{equation*}
\begin{aligned}
\ft \Ul(t, k)
&=
e^{-(\Im M_0) k t} \ft \Ul(0, k) 
- i \int_0^t e^{-(\Im M_0) k (t-t')} \ft {G_1}(t',k) dt'
\end{aligned}
\end{equation*}
for $t \in [0, T]$.
Then, it holds that
\begin{equation}
\begin{aligned}
\ft \Ul(0,k)
&=
e^{(\Im M_0) k T} \ft \Ul (T,k) 
+ i
\int_0^T e^{(\Im M_0) k t'} \ft {G_1}(t',k) dt'.
\end{aligned}
\label{Wk0}
\end{equation}

Fix $s > \frac 32$ and $0 < \dl < 1$. 
From Lemma \ref{lem:funda}, 
there exists 
\[
C_s := C (s, \|u\|_{L_T^\infty H^{s-1}}, \|v\|_{L_T^\infty H^{s-1}}, \|U\|_{L_T^\infty H^{s-1}}
, \|V\|_{L_T^\infty H^{s-1}}) > 0
\]
such that
\begin{equation}
\begin{aligned}
&\Big \| \jb{k}^{s-1+\dl} \int_0^T e^{(\Im M_0) k t'} P_- \ft {G_1} (t', k) dt' \Big\|_{\l^2}
\\
&\le \int_0^T \| \jb{k}^\dl e^{(\Im M_0) k t'} \jb{k}^{s-1} P_- \ft {G_1} (t', k)\|_{\l^2} dt'
\\
&\les \int_0^T ((\Im M_0) t')^{-\dl} \|G_1\|_{H^{s-1}} dt'
\\
&\le
\frac {T^{1-\dl}}{(\Im M_0)^\dl (1 - \dl)} C_s.
\end{aligned}
\label{noli}
\end{equation}

Since 
\[
\sup_{k < 0} \; \jb{k}^\dl e^{(\Im M_0) k T} < \infty,  
\]
we obtain
\begin{equation}
\begin{aligned}
&\| \jb{k}^{s-1+\dl} e^{(\Im M_0) k T} P_- \ft \Ul(T,k) \|_{\l^2}
\\
&\les
\| \jb{k}^{s-1} \ft \Ul(T,k) \|_{\l^2}
=
\| \Ul(T) \|_{H^{s-1}}
\le
C_s.
\end{aligned}
\label{lin}
\end{equation}
It follows from \eqref{Wk0}, \eqref{noli}, and \eqref{lin} that
\begin{equation}
\begin{aligned}
\|P_- \Ul(0)\|_{H^{s-1+\dl}}
\le
C_s.
\end{aligned}
\label{P+Ulk}
\end{equation}

By \eqref{gauge}, Lemma \ref{lem:bili3},
and \eqref{P+Ulk}, we obtain
\begin{equation*}
\begin{aligned}
\| P_- u_0 \|_{H^{s+\dl}}
&\les
\| P_- (e^{\Ld} U)(0) \|_{H^{s-1+\dl}}
\\
&\le
\| P_- (e^{\Ld} P_- U)(0) \|_{H^{s-1+\dl}}
+ 
\| P_- (e^{\Ld} P_0 U)(0) \|_{H^{s-1+\dl}}
\\
&\hspace*{80pt}+
\| P_- (e^{\Ld} P_+ U)(0) \|_{H^{s-1+\dl}}
\\
&\le
\| e^{\Ld(0)} \|_{H^{\frac 12+ \eps}} \| P_- \Ul (0)\|_{H^{s-1+\dl}}
+
\| e^{\Ld(0)} \|_{H^s} \| U(0) \|_{H^{\frac 12+\eps}}
\\
&\le
C_s
\end{aligned}
\end{equation*}
for $0 < \eps \ll 1$.
The same calculation in this proof yields that
\[
\| P_- v_0 \|_{H^{s+\dl}} \le C_s.
\]
This shows Theorem \ref{thm:NE}.
\end{proof}

Theorem \ref{thm:IP} can be proved by constructing 
an approximated sequence $\{(u_{0,n}, v_{0,n})\}_{n \in \N} \subset \H^s(\T)$.

\begin{proposition}
\label{prop:app} 
Let $s > \frac 32$. For $0<\dl<1$ and $(u_0, v_0) \in \H^s(\T)$, there exists a sequence
$\{(u_{0,n}, v_{0,n})\}_{n \in \N} \subset \H^s(\T)$ such that the followings hold:
\begin{enumerate}
\item $(u_{0,n}, v_{0,n})$ converges to $(u_0, v_0)$ in $\H^s(\T)$ as $n \to \infty$;
\item $\Im M (u_{0, n}, v_{0, n}) \neq 0$ for any $n \in \N$;
\item $(P_\pm u_{0,n}, P_\pm v_{0,n}) \notin \H^{s+\dl}(\T)$ for any $n \in \N$.
\end{enumerate}
\end{proposition}

\begin{proof}
We consider the two cases:
\begin{itemize}
\item Case 1: $\Im M(u_0, v_0) \neq 0$;
\item Case 2: $\Im M(u_0, v_0) = 0$.
\end{itemize}

Case 1: First, we consider the case $(u_0, v_0) \in \H^{s+\dl}(\T)$.
Set
\[
u_{0}^{(\l)}(x) = u_0 (x) + \frac 1\l \sum_{k \in 2^\N} k^{-s-\dl} (e^{i k x} + e^{-ikx})
\]
for $\l \in \N$.
It follows from $u_0 \in H^{s+\dl}(\T)$ that
\begin{equation*}
P_\pm u_0^{(\l)} = P_\pm (u_0^{(\l)} - u_0) + P_\pm u_0 \notin H^{s+\dl}(\T)
\end{equation*}
for any $\l \in \N$.

A simple calculation yields that
\begin{equation}
\| u_{0}^{(\l)} - u_0 \|_{H^s}
= \bigg( \frac 2{\l^2}  \sum_{k \in 2^{\N}} k^{-2\dl} \bigg)^{\frac 12}
\les \frac 1\l
\label{app2}
\end{equation}
for any $\l \in \N$.
Note that
\begin{align*}
\Big|
\Im \int_\T u_0^{(\l)} v_0 dx
- \Im \int_\T u_0 v_0 dx
\Big|
\le
\| u_{0}^{(\l)} - u_0 \|_{L^2} \|v_0\|_{L^2}.
\end{align*}
From \eqref{app2} and $s>\frac 32$,
we obtain that
\[
\lim_{\l \to \infty} \Im \int_\T u_{0}^{(\l)} v_0 dx
= \Im \int_\T u_0 v_0 dx.
\]
Hence, there exists $N \in \N$ such that $\Im M (u_{0}^{(\l)}, v_0) \neq 0$ for $\l \ge N$.
Thus, we can take $\{(u_{0, n}, v_{0, n})\}_{n \in \N} = \{(u_{0}^{(n+N)}, v_0)\}_{n \in \N}$.

Next, we consider the case $(u_0, v_0) \notin \H^{s+\dl}(\T)$.
Without loss of generality, we assume $u_0 \notin H^{s+\dl}(\T)$.
We separate this case into the following three cases:
\begin{itemize}
\item
Case 1-1:
$P_+ u_0 \notin H^{s+\dl}(\T)$ and $P_- u_0 \notin H^{s+\dl}(\T)$;

\item
Case 1-2:
$P_+ u_0 \in H^{s+\dl}(\T)$ and $P_- u_0 \notin H^{s+\dl}(\T)$;

\item
Case 1-3:
$P_+ u_0 \notin H^{s+\dl}(\T)$ and $P_- u_0 \in H^{s+\dl}(\T)$.
\end{itemize}

Case 1-1:
We can take $\{(u_{0,n}, v_{0,n})\}_{n \in \N} = \{(u_{0}, v_0)\}_{n \in \N}$.

Case 1-2:
We set
\[
u_{0, +}^{(\l)} (x) = u_0(x) + \frac 1\l \sum_{k \in 2^\N} k^{-s-\dl} e^{i k x}
\]
for $\l \in \N$.
The same argument as in the case $(u_0, v_0) \in \H^{s+\dl}(\T)$  shows that
we can take
$\{(u_{0,n}, v_{0,n})\}_{n \in \N} = \{(u_{0,+}^{(n+N)}, v_0)\}_{n \in \N}$ for some large $N \in \N$.

Case 1-3:
This case is similarly handled as in Case 1-2.

Case 2: Assume that $\Im M (u_0, v_0) = 0$.
We claim that there exists $\{(u_{0}^{(\l)}, v_{0}^{(\l)} ) \}_{\l \in \N} \subset \H^s(\T)$ such that
\begin{equation}
\| (u_0^{(\l)}, v_0^{(\l)}) - (u_{0}, v_{0}) \|_{\H^s} = \frac 1\l 
\quad \text {and} \quad \Im M (u_{0}^{(\l)}, v_{0}^{(\l)}) \neq 0.
\label{c:app}
\end{equation}

If $P_0 v_0 \neq 0$, we set
\[
(u_0^{(\l)}, v_0^{(\l)}) =
\left\{
\begin{aligned}
&\Big( u_0 + \frac 1\l, v_0 \Big), && \text {if $\Im P_0 v_0 \neq 0$}, \\
&\Big( u_0 + \frac i\l, v_0 \Big), && \text {if $\Im P_0 v_0 = 0$}.
\end{aligned}
\right.
\] 
Then, a direct calculation shows that
\[
\| (u_0^{(\l)}, v_0^{(\l)}) - (u_0 , v_0) \|_{\H^s} 
= \frac 1\l.
\]
Moreover, we obtain
\begin{equation*}
\Im M (u_0^{(\l)}, v_0^{(\l)})
=
\Im \int_\T \Big(u_0 v_0 + \frac 1\l v_0 \Big) \; dx 
= 
\frac 1\l \Im P_0 v_0
\neq 0
\end{equation*}
when $\Im P_0 v_0 \neq 0$.
Similarly, we obtain
\[
\Im M (u_0^{(\l)}, v_0^{(\l)})
=
\Im \int_\T \Big( u_0 v_0 + \frac i\l v_0 \Big) \; dx 
= 
\frac 1\l \Re P_0 v_0
\neq 0
\]
provided with $\Im P_0 v_0 = 0$.
Since $P_0 v_0 \neq 0$,
$\{(u_0^{(\l)}, v_0^{(\l)})\}_{\l \in \N}$ satisfies \eqref{c:app}.

If $P_0 u_0 \neq 0$, define
\[
(u_0^{(\l)}, v_0^{(\l)}) =
\left\{
\begin{aligned}
&\Big( u_0, v_0 + \frac 1\l \Big), && \text {if $\Im P_0 u_0 \neq 0$}, \\
&\Big( u_0, v_0 + \frac i\l \Big), && \text {if $\Im P_0 u_0 = 0$}.
\end{aligned}
\right.
\] 
The same argument as in the case of $\Im P_0 v_0 \neq 0$ implies that
$\{(u_0^{(\l)}, v_0^{(\l)})\}_{\l \in \N}$ satisfies \eqref{c:app}.

When $P_0 u_0 = 0$ and $P_0 v_0 = 0$ hold,
we take $\{(u_0^{(\l)}, v_0^{(\l)})\}_{\l \in \N}$ as
\[
(u_0^{(\l)}, v_0^{(\l)}) = \Big( u_0 + \frac 1{\sqrt \l}, v_0 + \frac i {\sqrt \l} \Big).
\]
Then, we have 
\[
\| (u_0^{(\l)}, v_0^{(\l)}) - (u_0, v_0) \|_{\H^s} = \frac 1\l.
\]
Moreover, 
\[
\begin{aligned}
\Im \int_\T  u_0^{(\l)} v_0^{(\l)} dx 
&=
\Im \int_\T u_0 v_0 dx + \frac 1{\sqrt \l}\Im P_0 v_0 + \frac 1{\sqrt \l} \Re P_0 u_0 
+ \frac 1{\l}  
\\
&= \frac 1{\l} \neq 0.
\end{aligned}
\]
Therefore, \eqref{c:app} holds.

The same argument as in Case 1 with \eqref{c:app} yields that
for any $\l \in \N$,
there exists $\{(u_0^{(\l, m)}, v_0^{(\l, m)}) \}_{m \in \N}$ such that
the following conditions hold:
\begin{enumerate}
\item $(u_0^{(\l, m)}, v_0^{(\l, m)})$ converges to $(u_0^{(\l)}, v_0^{(\l)})$ 
in $\H^s(\T)$ as $m \to \infty$;
\item $\Im M (u_0^{(\l, m)}, v_0^{(\l, m)}) \neq 0$ for any $m \in \N$;
\item $(P_\pm u_0^{(\l, m)}, P_\pm v_0^{(\l, m)}) \notin \H^{s+\dl}(\T)$ for any $m \in \N$.
\end{enumerate}
In particular,
it follows from the proof of the Case 1 that
\begin{equation}
\| (u_0^{(\l, m)}, v_0^{(\l, m)}) - (u_0^{(\l)}, v_0^{(\l)}) \|_{\H^s}
\les
\frac 1m.
\label{app3}
\end{equation}
The triangle inequality, \eqref{c:app}, and \eqref{app3} imply that
\[
\begin{aligned}
&\| (u_0^{(n,n)}, v_0^{(n,n)}) - (u_0, v_0) \|_{\H^s}
\\
&\le
\| (u_0^{(n, n)}, v_0^{(n, n)}) - (u_0^{(n)}, v_0^{(n)}) \|_{\H^s}
+ \| (u_0^{(n)}, v_0^{(n)}) - (u_0, v_0) \|_{\H^s}
\\
&\les \frac 1n.
\end{aligned}
\]
Thus, we can take $\{ (u_{0, n}, v_{0, n}) \}_{n \in \N}$ as 
$ \{ (u_0^{(n,n)}, v_0^{(n,n)}) \}_{n \in \N}$.
This concludes the proof.
\end{proof}

\section{The proof of Theorem \ref{thm:WP}}
\label{SEC:WP}
In this section, we prove the well-posedness for \eqref{CLL} in $\H_0^s(\T)$ 
defined in \eqref{subset} by the contraction mapping argument.
For simplicity, we consider a solution to \eqref{CLL} forward in time.

As in Section \ref{SEC:IP}, we set
\begin{equation*}
\begin{aligned}
\uu &:= \dx u, \quad \fv := \dx v,
\\
\Ld &:= - \frac i2  \dx^{-1} (u v),
\quad U := e^{-\Ld} \uu, \quad V:= e^\Ld \fv.  
\end{aligned}
\end{equation*}
Then,
$(U, V)$ satisfies
\begin{equation}
\left\{
\begin{aligned}
&i \dt U - \dx^2 U  
= - (\Im M(u,v)) \dx U + G_1,
\\
&i \dt V + \dx^2 V 
= - (\Im M(u,v)) \dx V 
+ G_2,
\end{aligned}
\right.
\label{CLLg2}
\end{equation}
where $G_1$ and $G_2$ is defined in \eqref{d:G}.

\begin{proof}[Proof of Theorem \ref{thm:WP}]
In this proof, we assume that $(u_0, v_0) \in \H_0^s(\T)$.

Fix $s > \frac 32$. Let $(U_0, V_0) \in \H^{s-1}(\T)$.
We consider the following integral equations which correspond to \eqref{CLL} and \eqref{CLLg2}:
\begin{align}
u(t) &= e^{- i t \dx^2} u_0 
+ \int_0^t e^{- i (t - t') \dx^2} (u v e^{\Ld} U) (t')  dt',
\label{intu}
\\
v(t) &= e^{i t \dx^2} v_0 
+ \int_0^t e^{i (t - t') \dx^2} (u v e^{-\Ld} V) (t')  dt',
\label{intv}
\\
U(t) &= e^{- i t \dx^2 + i (\int_0^t \Im M(u(t'), v(t')) dt') \dx} U_0 
\notag
\\
&\qquad- 
i \int_0^t e^{- i (t - t') \dx^2 + i (\int_{t'}^t \Im M(u(\tau), v(\tau)) d\tau) \dx} G_1(t') dt',
\label{intU}
\\
V(t) &= e^{ i t \dx^2 + i (\int_0^t \Im M(u(t'), v(t')) dt') \dx} V_0
\notag 
\\
&\qquad- 
i \int_0^t e^{i (t - t') \dx^2 + i (\int_{t'}^t \Im M(u(\tau), v(\tau)) d\tau) \dx} G_2(t') dt'.
\label{intV}
\end{align}
We write the right-hand sides of \eqref{intu}, \eqref{intv}, \eqref{intU}, and \eqref{intV} as
$\G_1(u,v,U,V)$, $\cdots$, $\G_4(u,v,U,V)$, respectively.

We define
\[
A^s := \H_0^s(\T) \times \H^s(\T).
\]
In the following,
we write $((f, g), (F, G)) \in A^s$ as $(f,g,F,G)$ for simplicity.
We also define the norm in $A^s$ as
\[
\|(f, g, F, G)\|_{A^s} := \|(f, g) \|_{\H^s} + \|(F, G)\|_{\H^s}
\]
for $(f,g,F,G) \in A^s$.

Set
\[
K:= 2\|(u_0, v_0, U_0, V_0)\|_{A^{s-1}}.
\]
For $T > 0$, we define 
\[
\begin{aligned}
B_T 
:= 
\{ 
(u, v, U, V) \in C([0,T]; A^{s-1}) \mid 
\|(u,v,U,V)\|_{L_T^\infty A^{s-1}} \le K
\}.
\end{aligned}
\]
A direct calculation with \eqref{es:dxu} yields that there exists $C(s,K) > 0$ which is independent of $t$ such that 
\begin{equation}
\begin{aligned}
&\| \G_1 (u,v,U,V) \|_{L_T^\infty H^{s-1}}
\\
&\le
\|u_0\|_{H^{s-1}} + \int_0^T \| (u v e^{\Ld} U) (t') \|_{H^{s-1}} dt'
\\
&\les
\|u_0\|_{H^{s-1}} + T \|u\|_{L_T^\infty H^{s-1}} \|v\|_{L_T^\infty H^{s-1}} 
\|U\|_{L_T^\infty H^{s-1}}
\\
&\le
\|u_0\|_{H^{s-1}} 
+ C(s, K) T
\end{aligned}
\label{solu}
\end{equation}
for $(u,v,U,V) \in B_T$.
Similarly, we have
\begin{equation}
\begin{aligned}
\| \G_2 (u,v,U,V) \|_{L_T^\infty H^{s-1}}
&\le
\|v_0\|_{H^{s-1}} + \int_0^t \| (u v e^{- \Ld} V) (t') \|_{H^{s-1}} dt'
\\
&\le
\|v_0\|_{H^{s-1}} + C(s, K) T
\end{aligned}
\label{solv}
\end{equation}
for $(u,v,U,V) \in B_T$.

It follows from $\Im M(u,v) = 0$ for $(u,v,U,V) \in B_T$ and Lemma \ref{lem:funda} that
\begin{equation}
\begin{aligned}
&\|\G_3(u,v,U,V)\|_{L_T^\infty H^{s-1}}
\\
&\le
\sup_{t \in [0,T]} \Big( \|U_0\|_{H^{s-1}}+
\Big \| \int_0^t e^{- i (t - t') \dx^2 } G_1(t') dt' \Big\|_{H^{s-1}} \Big)
\\
&\le
\|U_0\|_{H^{s-1}}
+
C(s, K) T
\end{aligned}
\label{solU}
\end{equation}
for $(u,v,U,V) \in B_T$.
The same calculation as in \eqref{solU} yields that for $(u,v,U,V) \in B_T$,
\begin{equation}
\begin{aligned}
\| \G_4 (u,v,U,V) \|_{L_T^\infty H^{s-1}}
\le \|V_0\|_{H^{s-1}} 
+ C(s, K) T.
\end{aligned}
\label{solV}
\end{equation}

By \eqref{solu}--\eqref{solV}, we obtain that 
\[
\begin{aligned}
&\G (u,v,U,V) 
\\
&:= (\G_1(u,v,U,V), \G_2(u,v,U,V), \G_3(u,v,U,V), \G_4(u,v,U,V))
\end{aligned}
\] 
is a map on $B_T$ by taking $T > 0$ sufficiently small.
A similar calculation as in \eqref{solu}--\eqref{solV} shows that for 
$(u,v, U, V)$, $(u', v', U', V') \in B_T$, there exists 
$C'(s, K) > 0$ such that
\[
\begin{aligned}
&\| \G((u, v, U, V)) - \G ((u',v',U',V')) \|_{L_T^\infty A^{s-1}}
\\
&\le
C'(s, K) T \|(u,v, U, V) - (u', v', U', V')\|_{L_T^\infty A^{s-1}}.
\end{aligned}
\] 
Hence, $\G$ is a contraction map on $B_T$ by taking $T > 0$ small enough.
Thus, \eqref{intu}--\eqref{intV} has a solution in $C([0,T]; A^{s-1}(\T))$.
The uniqueness and the continuous dependence follow from a standard argument.
We thus omit the details here.

Set 
\[
U_0 := e^{-\Ld(0)} \dx u_0, \quad V_0:= e^{\Ld(0)} \dx v_0.
\] 
Since $(U_0, V_0) \in \H^{s-1}(\T)$, 
there exists a solution $(u, v, U, V) \in C([0,T]; A^{s-1})$ to \eqref{intu}--\eqref{intV} with the initial data
$(u_0,v_0,U_0,V_0) \in A^{s-1}$. 
Then, \eqref{CLLg2} and the uniqueness of \eqref{intu}--\eqref{intV} imply that 
\[
U = e^{-\Ld} \dx u, \quad V = e^{\Ld} \dx v.
\]
In particular, $e^{-\Ld} \dx u \in C([0,T]; H^{s-1}(\T))$. 
Hence, with \eqref{p:eLd}, we obtain
\[
\begin{aligned}
\|u\|_{L_T^\infty H^s}
&= 
\|(1+\dx^2) u \|_{L_T^\infty H^{s-2}}
\\
&\le
\|u\|_{L_T^\infty H^{s-2}} + \| e^{\Ld} e^{-\Ld} \dx u \|_{L_T^\infty H^{s-1}}
\\
&\les
\|u\|_{L_T^\infty H^{s-2}} + \| e^{\Ld} \|_{L_T^\infty H^{s-1}} \| e^{-\Ld} \dx u \|_{L_T^\infty H^{s-1}} 
\\
&\le
C(s, \|u\|_{L_T^\infty H^{s-1}}, \|v\|_{L_T^\infty H^{s-1}}, \|U\|_{L_T^\infty H^{s-1}}).
\end{aligned}
\]
A similar calculation above shows that
\[
\|v\|_{L_T^\infty H^s} 
\le 
C(s, \|u\|_{L_T^\infty H^{s-1}}, \|v\|_{L_T^\infty H^{s-1}}, \|V\|_{L_T^\infty H^{s-1}}).
\]
Therefore, we conclude that the equation \eqref{CLL} has a unique solution 
\[
(u, v) \in C([0,T]; \H_0^s(\T)).
\]
The continuity of the flow map for \eqref{CLL} in $\H_0^s(\T)$
follows immediately from continuous dependence for \eqref{intu}--\eqref{intV} in $A^{s-1}$.
\end{proof}
\mbox{}

\noindent
{\bf 
Acknowledgements.}
This work was
supported by JST SPRING Grant Number JPMJSP2132.
The author would like to express his deepest gratitude to Mamoru Okamoto for his support and encouragement throughout this research.


\begin{thebibliography}{99}
\bibitem{AKNS74}
M. J. Ablowitz, D. J. Kaup, A. C. Newell, H. Segur,
{\it The inverse scattering transform-Fourier analysis for nonlinear problems},
Studies in Appl. Math. \textbf{53} (1974), no. 4, 249–315.

\bibitem{AbMu13}
M. J. Ablowitz, Z. H. Musslimani,
{\it Integrable Nonlocal Nonlinear Schr\"odinger Equation},
Phys. Rev. Lett. \textbf{110} (2013), no. 6-8, 064105.

\bibitem{AbMu17}
M. J. Ablowitz, Z. H. Musslimani, 
{\it Integrable Nonlocal Nonlinear Equations},
Stud. Appl. Math. \textbf{139} (2017), no. 1, 7–59.

\bibitem{AmSi15}
D. Ambrose, G. Simpson,
{\it Local existence theory for derivative nonlinear Schr\"odinger equations with noninteger power nonlinearities},
SIAM J. Math. Anal. \textbf{47} (2015), no. 3, 2241–2264.

\bibitem{BaPe22}
H. Bahouri, G. Perelman,
{\it Global well-posedness for the derivative nonlinear Schr\"odinger equation},
Invent. Math. \textbf{229} (2022), no. 2, 639–688.

\bibitem{BaPe24}
H. Bahouri, G. Perelman,
{\it Global well-posedness for the derivative nonlinear Schr\"odinger equation with periodic boundary condition},
Int. Math. Res. Not. IMRN 2024, no. 24, 14479–14518.

\bibitem{CLL79}
H. H. Chen, Y. C. Lee, C. S. Liu, 
{\it Integrability of nonlinear Hamiltonian system by inverse scattering method},
Phys. Scr. \textbf{20} (1979) 490--492.

\bibitem{CLW23}
J. Chen, Y. Lu, B. Wang,
{\it Global Cauchy problems for the nonlocal (derivative) NLS in $E^s_\s$},
J. Differential Equations \textbf{344} (2023), 767–806.

\bibitem{GrHe08}
A. Gr\"unrock, S. Herr,
{\it Low regularity local well-posedness of the derivative nonlinear Schr\"odinger equation with periodic initial data},
SIAM J. Math. Anal. \textbf{39} (2008), no. 6, 1890–1920.

\bibitem{GuWu17}
Z. Guo, Y. Wu,
{\it Global well-posedness for the derivative nonlinear Schr\"odinger equation in $H^{\frac12}(\R)$}, 
Discrete Contin. Dyn. Syst. \textbf{37} (2017), no. 1, 257–264.

\bibitem{HKNV26}
B. Harrop-Griffiths, R. Killip, M. Ntekoume, M. Visan,
{\it Global well-posedness for the derivative nonlinear Schr\"odinger equation in $L^2(\R)$},
J. Eur. Math. Soc. (JEMS) \textbf{28} (2026), no. 2, 843–924.

\bibitem{Ha93}
N. Hayashi,
{\it The initial value problem for the derivative nonlinear Schr\"odinger equation in the energy space},
Nonlinear Anal. \textbf{20} (1993), no. 7, 823–833.

\bibitem{HaOz92}
N. Hayashi, T. Ozawa,
{\it On the derivative nonlinear Schr\"odinger equation},
Phys. D \textbf{55} (1992), no. 1--2, 14--36.

\bibitem{He06}
S. Herr,
{\it On the Cauchy problem for the derivative nonlinear Schr\"odinger equation with periodic boundary condition},
Int. Math. Res. Not. (2006), Art. ID 96763, 33 pp.


%
\bibitem{JLPS20}
R. Jenkins, J. Liu, P. Perry, C. Sulem,
{\it Global existence for the derivative nonlinear Schr\"odinger equation with arbitrary spectral singularities},
Anal. PDE \textbf{13} (2020), no. 5, 1539–1578.


%

\bibitem{KNV23}
R. Killip, M. Ntekoume, M. Visan,
{\it On the well-posedness problem for the derivative nonlinear Schr\"odinger equation},
Anal. PDE \textbf{16} (2023), no. 5, 1245–1270.

\bibitem{KiTs18}
N. Kishimoto, Y. Tsutsumi,
{\it Ill-posedness of the third order NLS equation with Raman scattering term},
Math. Res. Lett. \textbf{25} (2018), no. 5, 1447--1484.

\bibitem{KoOk25b}
T. Kondo, M. Okamoto,
{\it Well- and ill-posedness of the Cauchy problem for semi-linear Schr\"odinger equations on the torus}, to appear in Funkcial. Ekvac.

\bibitem{Ku84}
A. Kundu,
{Landau-Lifshitz and higher-order nonlinear systems gauge generated from nonlinear Schrödinger type equations},
J. Math. Phys. \textbf{25} (1984), no. 12, 3433–3438.


\bibitem{MMW07}
J. Moses, B. Malomed, F. Wise, 
{\it Self-steepening of ultrashort optical pulses without self-phase-modulation},
Phys. Rev. A \textbf{76} (2007), 1-4.

\bibitem{Mo17}
R. Mosincat,
{\it Global well-posedness of the derivative nonlinear Schr\"odinger equation with periodic boundary condition in $H^{\frac12}$},
J. Differential Equations \textbf{263} (2017), no. 8, 4658–4722.

\bibitem{NaWa25}
K. Nakanishi, B. Wang,
{\it Global wellposedness of general nonlinear evolution equations for distributions on the Fourier half space},
J. Funct. Anal. \textbf{289} (2025), no. 8, Paper No. 111004, 70 pp.


\bibitem{SSZ19}
Y. Shi, S. F. Shen, S. L. Zhao,
{\it Solutions and connections of nonlocal derivative nonlinear Schr\"odinger equations}, 
Nonlinear Dyn. \textbf{95} (2019), 1257–1267.

\bibitem{Ta99}
H. Takaoka,
{\it Well-posedness for the one-dimensional nonlinear Schr\"odinger equation with the derivative nonlinearity},
Adv. Differential Equations \textbf{4} (1999), no. 4, 561–580.


\bibitem{TsFu80}
M. Tsutsumi, I. Fukuda,
{\it On solutions of the derivative nonlinear Schr\"odinger equation. Existence and uniqueness theorem},
Funkcial. Ekvac. \textbf{23} (1980), no. 3, 259–277.

\bibitem{TsFu81}
M. Tsutsumi, I. Fukuda,
{\it On solutions of the derivative nonlinear Schr\"odinger equation. II},
Funkcial. Ekvac. \textbf{24} (1981), no. 1, 85–94.
\end{thebibliography}
\end{document}